\documentclass{amsart}
\usepackage{amsthm,hyperref}
\usepackage{amsmath,amssymb,MnSymbol}
\usepackage{cite}
\usepackage{tikz}
\numberwithin{equation}{section}
\newtheorem{thm}{Theorem}[section]
\newtheorem{con}[thm]{Conjecture}

\newtheorem{prop}[thm]{Proposition}
\newtheorem{problem}[thm]{Problem}

\theoremstyle{definition}

\DeclareMathOperator{\Mon}{Mon}
\DeclareMathOperator{\lex}{lex}
\DeclareMathOperator{\Sym}{Sym}
\DeclareMathOperator{\Ann}{Ann}

\begin{document}
\title[Quadratic Ideals and EGH in Six Variables]{Quadratic Ideals in Six Variables and the Eisenbud--Green--Harris Conjecture}
\author[A.Abedelfatah]{Abed Abedelfatah}

\address{Department of Mathematics, Braude College of Engineering Karmiel, 2161002 Karmiel, Israel}
\email{abed@braude.ac.il}
\begin{abstract}
In this paper, we study the Eisenbud--Green--Harris (EGH) conjecture
for ideals generated by quadrics.
We establish a sharp lower bound for the dimension of the cubic component
of an ideal generated by a regular sequence of six quadrics and two
additional quadrics in six variables.
Furthermore, we prove the Eisenbud--Green--Harris conjecture for almost
complete intersections of quadrics in six variables.
\end{abstract}

\maketitle

\section{Introduction}\label{sec:intro}
The Eisenbud-Green-Harris (EGH) conjecture states that a graded ideal in a polynomial ring $K[x_1,\dots,x_n]$ over a field $K$ that contains a regular sequence $f_1,\dots,f_n$ with degrees $2\leq a_1\leq \cdots \leq a_n$, has the same Hilbert function as a graded ideal containing $x_1^{a_1},\dots,x_n^{a_n}$. 

The Eisenbud-Green-Harris (EGH) conjecture extends Macaulay's theorem,  
which characterizes possible Hilbert functions of graded ideals through  
lexicographic (lex) ideals.  

The conjecture is inspired by the Cayley-Bacharach Theorem and has  
significant implications across various fields of mathematics.  

If proven true, it would support the generalized Cayley-Bacharach conjecture  
in algebraic geometry, expand the Clements-Lindström theorem in combinatorics,  
and provide strong evidence for the Lex-plus-powers conjecture in commutative algebra.  

By linking geometric properties, combinatorial principles, and algebraic structures,  
the EGH conjecture offers a comprehensive framework for understanding the behavior  
of homogeneous ideals and their Hilbert functions.

The conjecture is known true when $n=2$ \cite{richert} (the case $n=1$ is trivial). To date, the conjecture is widely open. See \cite{gunt1} for a good survey on the EGH conjecture.

The following problem was proposed by Chen \cite{chen} and has attracted considerable attention,
as it is closely related to the EGH conjecture.

\begin{problem}[Problem~3.6,\cite{chen}]\label{prob:chen36}
Let $S = K[x_1,\dots,x_n]$ be a polynomial ring and $f_1,\dots,f_n \in S_2$ form a regular sequence of quadrics with $n \geq 3$,
and let $g,h \in S_2$ be two quadrics such that
\[
\dim_K(f_1,\dots,f_n,g,h)_2 = n+2.
\]
Is it true that
\[
\dim_K(f_1,\dots,f_n,g,h)_3 \;\geq\;
\dim_K(x_1^2,\dots,x_n^2, x_1x_2, x_1x_3)_3 \;=\; n^2+2n-5 \, ?
\]
\end{problem}

In Section~\ref{sec:prelim}, we introduce the necessary background, notation,
and preliminary results that will be used throughout the paper. 

In Section~\ref{sec:discussion}, we examine previous approaches related to
Problem~\ref{prob:chen36}. In particular, we show that the attempted proof in \cite{gunt2}
relies on Proposition~3.12, which is flawed, and therefore does not resolve Chen’s question. 
On the other hand, we recall that the case $n=5$ of the EGH conjecture
was correctly established in \cite[Theorem 4.14]{caviglia}. 

In Section~\ref{sec:mainproof}, we present our new contributions:
first, we give a complete solution to Problem~\ref{prob:chen36} when $n=6$;
second, we prove the Eisenbud--Green--Harris conjecture for almost complete intersections
of quadrics in six variables.
\section{Preliminaries and notations}\label{sec:prelim}
A proper ideal $I$ in $S=K[x_1,\dots,x_n]$ is called \emph{graded} or \emph{homogeneous} if it has a system of homogeneous generators. We define the Hilbert function of $I$ as 
\[
\mathcal{H}_{S/I}(t)=\dim_K(S_t/I_t)=\dim_K S_t-~\dim_K I_t. 
\]
where $S_t$ and $I_t$ are the graded $t$ components of $S$ and $I$, respectively.

For simplicity, we denote the dimension of a $K$-vector space $V$ by $|V|$ instead of $\dim_{K}V$. For a $K$-vector space $V\subseteq S_d$, where $d\geq 0$, we denote by $S_1V$ the $K$-vector space spanned by $\{x_iv:~1\leq i\leq n~\wedge~v\in V\}$. Throughout this paper $\textbf{\underline{a}}=(a_1,\dots,a_n)\in \mathbb{Z}^n$, where $2\leq a_1\leq\cdots\leq a_n$.

We define the \emph{lex order} on $\Mon(S)$ by setting $\textbf{x}^b=x_1^{b_1}\cdots x_n^{b_n}<_{\lex}x_1^{c_1}\cdots x_n^{c_n}=\textbf{x}^c$ if either $\deg(\textbf{x}^b)<\deg(\textbf{x}^c)$ or $\deg(\textbf{x}^b)=\deg(\textbf{x}^c)$ and $b_i<c_i$ for the first index $i$ such that $b_i\neq c_i$. We recall the definitions of lex ideal and lex-plus-powers ideal.
A graded ideal is called \emph{monomial} if it has a system of monomial generators. A monomial ideal $I\subseteq S$ is called \emph{lex}, if whenever $I\ni z<_{\lex}w$, where $w,z$ are monomials of the same degree, then $w\in I$.
A monomial ideal $I$ is \emph{$\textbf{\underline{a}}$-lex-plus-powers} if there exists a lex ideal $L$ such that $I=\langle x_1^{a_1},\dots,x_n^{a_n}\rangle+L$.

In 1927, F.Macaulay proved that if $I$ is a graded ideal in $S$, then there exists a lex ideal $L$ such that $L$ has the same Hilbert function as $I$ \cite{macaulay}; i.e., every Hilbert function in $S$ is attained by a lex ideal. In \cite{clements}, Clements and Lindstr{\"o}m proved that every Hilbert function in $S/\langle x_1^{a_1},\dots,x_n^{a_n}\rangle$ is attained by an $\textbf{\underline{a}}$-lex-plus-powers ideal.

Let $f_1,\dots,f_n$ be a regular sequence in $S$ such that $2\leq a_1=\deg(f_1)\leq \cdots\leq a_n=\deg(f_n)$. A well-known result says that $\langle f_1,\dots,f_n\rangle$ has the same Hilbert function as $\langle x_1^{a_1},\dots,x_n^{a_n}\rangle$. It is natural to ask what happens if $I\subseteq S$ is a homogeneous ideal containing a regular sequence in fixed degrees. This question brings us to the Eisenbud-Green-Harris conjecture, denoted by EGH.

\begin{con}[EGH \cite{eisenbud}]\label{1}{\ \\}
If $I$ is a homogeneous ideal in $S$ containing a regular sequence $f_1,\dots,f_n$ of degrees $\deg(f_i)=a_i$, where $2\leq a_1\leq\cdots\leq a_n$, then $I$ has the same Hilbert function as an ideal containing $x_1^{a_1},\dots,x_n^{a_n}$.
\end{con}

Since the Hilbert function of a monomial ideal is independent of the field $K$,
by an extension of $K$ we may assume that $K$ is infinite and algebraically closed.\\
By Clements-Lindstr{\"o}m's theorem, the ideal $I$ has the same Hilbert function as an ideal containing $x_1^{a_1},\dots,x_n^{a_n}$ if and only if for all $t\geq 0$, there is an ideal $L$ containing $x_1^{a_1},\dots,x_n^{a_n}$ such that $|L_t|=|I_t|$ and $|L_{t+1}|\leq |I_{t+1}|$.
Note that if $I$ is a graded ideal containing a regular sequence $f_1,\dots,f_n$, then $I_t=S_t$ for all $t> N=\sum_{i=1}^{n}(a_i-1)$ and $I$ satisfies the EGH Conjecture if and only if, for all $0\leq t\leq N$, there is an ideal $L$ containing $x_1^{a_1},\dots,x_n^{a_n}$ such that $|I_t|=|L_t|$ and $|L_{t+1}|\leq |I_{t+1}|$.

In our argument, we will use the following results of the author.

\begin{prop}[Proposition~3.1,\cite{abed2}]\label{2}{\ \\}
Let $I$ be a graded ideal in $S$ containing $P$. Then $I$ satisfies the EGH conjecture if and only if $(P:I)$ satisfies the EGH Conjecture.
\end{prop}

\begin{thm}[Theorem~3.2,\cite{abed}]\label{3}{\ \\}
Let $I$ be a graded ideal in $S$ containing a regular sequence $$f_1,~\dots,f_{n-1},~f_n=q_1\cdots q_s$$ of degrees $\deg(f_i)=a_i$ such that $q_i\in S_1$ for all $1\leq i\leq s$. If the image of the sequence $f_1,\dots,f_{n-1}$ in $S/\langle q_i\rangle$ satisfies the EGH conjecture for all $i$, then $I$ has the same Hilbert function as a graded ideal in $S$ containing $x_1^{a_1},\dots,x_n^{a_n}$.
\end{thm}

We denote by $P$ the ideal generated by a regular sequence $f_1,f_2,\dots,f_n$ of degrees $a_1,\dots,a_n$ and let $A$ be the graded Gorenstein Artin $K$ algebra $S/P$. We also need the following useful proposition on Hilbert functions under liaison.

\begin{prop}[Theorem~3,\cite{Davis}]\label{6}{\ \\}
Let $I$ be a graded ideal in $S$ containing $P$ and $N=\sum_{i=1}^{n}(a_i-1)$, then for all $0\leq j\leq N$
\[
\mathcal{H}_{S/I}(j)=\mathcal{H}_A(j)-\mathcal{H}_{S/(P:I)}(N-j)
\]
\end{prop}

\section{Analysis of Proposition~3.12 from \cite{gunt2}} \label{sec:discussion}
In this section, we examine Proposition~3.12 from \cite{gunt2}, which concerns the existence of a quadric with a nontrivial linear annihilator in the context of the Eisenbud–Green–Harris conjecture.
While the proposition supports the conjecture in a specific setting, we identify a step in the proof that requires further justification. We present an explicit construction showing that the conclusion does not hold in general, thus suggesting that the result may need to be refined.

\subsection*{Gap in Proposition~3.12 of \cite{gunt2}}

Let \( f, g \in A_2 \) where 
\[
A = S/P, \quad S = K[x_1, \dots, x_n],
\]
and \( P = (f_1, \dots, f_n) \subset S \) is an ideal generated by a regular sequence 
of quadrics \( f_1, \dots, f_n \).  
Assume that \( \dim(fA_1) = \dim(gA_1) = n \).  
Then \cite{gunt2}, Proposition~3.12 asserts:

\emph{There exists a combination \( 0\neq h \in \langle f, g \rangle_2 \) such that
\(\dim Ann_{A_1}(h) \ge 1\).}

However, this proposition does not hold under the stated assumptions. 
To see this concretely, consider the ring
\[
A = K[x_1, x_2, x_3, x_4, x_5]/(x_1^2, x_2^2, x_3^2, x_4^2, x_5^2)
\]
and the quadratic forms
\[
g = x_1x_2 + x_3x_4, \qquad
h = x_5x_1 + x_2x_3.
\]
Let 
\[
q = \alpha g + \beta h \in A_2,
\]
where \( \alpha, \beta \in K \) are not both zero. 
Suppose that there exists a nontrivial linear form
\[
\ell = c_1x_1 + c_2x_2 + c_3x_3 + c_4x_4 + c_5x_5 \in A_1
\]
such that 
\[
\ell q = 0 \quad \text{in } A.
\]

Expanding \(\ell q\) and discarding all terms with \(x_i^2 = 0\), 
one obtains the system:
\begin{align*}
x_1x_2x_3: & \quad \alpha c_3 + \beta c_1 = 0, \\
x_1x_2x_4: & \quad \alpha c_4 = 0, \\
x_1x_2x_5: & \quad \alpha c_5 + \beta c_2 = 0, \\
x_1x_3x_4: & \quad \alpha c_1 = 0, \\
x_1x_3x_5: & \quad \beta c_3 = 0, \\
x_1x_4x_5: & \quad \beta c_4 = 0, \\
x_2x_3x_4: & \quad \alpha c_2 + \beta c_4 = 0, \\
x_2x_3x_5: & \quad \beta c_5 = 0, \\
x_3x_4x_5: & \quad \alpha c_5 = 0.
\end{align*}

If \( \alpha \neq 0 \), the equations immediately give
\( c_1 = c_2= c_3 = c_4 = c_5 = 0 \).  
Similarly, if \( \beta \neq 0 \), then
\( c_1=c_2=c_3 = c_4 = c_5 = 0 \). 
In all cases, 
\[
c_1 = c_2 = c_3 = c_4 = c_5 = 0,
\]
so that \( \ell = 0 \). 

This contradicts the conclusion of Proposition~3.12 in \cite{gunt2}.  

The proof of Proposition~3.12 implicitly identifies the two image spaces
\(fA_1\) and \(gA_1\). This identification is not justified by the stated
hypotheses. The equality \(fA_1=gA_1\) appears explicitly in Lemma~3.10 of
the same paper. Under this additional hypothesis, the argument gives the
following corrected statement:

\textit{
Let \( f, g \in A_2 \) with \( fA_1 = gA_1 \) and \( \dim(fA_1) = n \). 
Then there exists \( 0\neq h \in \langle f, g \rangle_2 \) such that 
\( \dim Ann_{A_1}(h) \ge 1 \).
}

This gap affects the argument based on Proposition~3.12 in \cite{gunt2}. Nevertheless, the Eisenbud--Green--Harris conjecture for five quadrics was independently established in \cite[Theorem~4.14]{caviglia}.

\section{New results for the case $n=6$} \label{sec:mainproof}
We present two results for the case $n=6$. First, we solve Problem~\ref{prob:chen36} by establishing a sharp lower bound on the dimension in degree $3$. Then, we prove the Eisenbud–Green–Harris conjecture for almost complete intersections of quadrics.
We begin by recalling the following result due to Francisco (see Corollary 5.2 of \cite{fran}).

\begin{prop}\label{4}
Let $I=(f_1,\dots,f_n,g)\subseteq S=K[x_1,\dots,x_n]$, where
$f_1,\dots,f_n$ form a regular sequence of quadrics and
$2\leq d=\deg(g)\leq n$. Set $P=(f_1,\dots,f_n)$. Then
\[
\dim_K I_{d+1}\geq
\dim_K\bigl(x_1^2,\dots,x_n^2,x_1x_2\cdots x_d\bigr)_{d+1}
=
\binom{n+d}{d+1}-\binom{n}{d+1}+n-d.
\]
Equivalently,
\[
\dim_K\bigl(S_1g\cap P_{d+1}\bigr)\leq d.
\]
\end{prop}
\begin{thm}\label{5}
    Let $S = K[x_1,\dots,x_6]$ be a polynomial ring and $f_1,\dots,f_6 \in S_2$ form a regular sequence of quadrics,
and let $g,h \in S_2$ be two quadrics such that
$\dim_K(f_1,\dots,f_6,g,h)_2 \;= 8$.
Then
\[
\dim_K(f_1,\dots,f_6,g,h)_3 \;\geq\;
\dim_K(x_1^2,\dots,x_6^2, x_1x_2, x_1x_3)_3 \;=\; 43.
\]
\end{thm}
\begin{proof}
Set $P=(f_1,\dots,f_6)$, $I=P+(g,h)$, and $A=S/P$.
If $I$ or $(P:I)$ contains a reducible quadric, then the theorem is true by Proposition~\ref{2} and Theorem~\ref{3}, since the EGH conjecture is true in the case $n=5$ where the degrees of the regular sequence equal $2$. Assume that neither $I$ nor $(P:I)$ contains a reducible quadric. We have two cases:

\textbf{Case 1:} There exists a nonzero element $q\in (g,h)_2$ such that $\Ann_{A_1}(q)\neq 0$.\\
Without loss of generality, assume $q=g$. We have to show that
\[
\dim_K\bigl(S_1g\cap P_3\bigr)
+
\dim_K\bigl(S_1h\cap (P+(g))_3\bigr)
\leq 5.
\]
By Proposition~\ref{4}, applied with $d=2$,
$\dim_K(S_1g\cap P_3)\leq 2$. Thus it is sufficient to prove that
$\dim_K(S_1h\cap(P+(g))_3)\leq 3$. Assume, to the contrary, that this dimension is at least $4$. Then there exist linearly independent linear forms $\ell_1,\dots,\ell_4\in S_1$ such that $\ell_i h\in P+(g)$ for all $1\leq i\leq 4$. Choose
$0\neq\ell\in\Ann_{A_1}(g)$, so that $\ell g\in P$. Multiplying the relations $\ell_i h\in P+(g)$ by $\ell$ gives
\[
\ell_i(\ell h)\in P_4
\qquad\text{for all }1\leq i\leq 4.
\]
If $\ell h\notin P$, Proposition~\ref{4}, applied to the cubic $\ell h$, gives
\[
\dim_K\bigl(S_1(\ell h)\cap P_4\bigr)\leq 3,
\]
contradicting the four linearly independent elements $\ell_i(\ell h)$ above. Hence $\ell h\in P$, and therefore $\ell\in(P:I)$. Consequently, $\ell m\in(P:I)_2$ for every $m\in S_1$, so $(P:I)$ contains a reducible quadric, a contradiction.

\textbf{Case 2:} For every nonzero element $q\in (g,h)_2$, $\Ann_{A_1}(q)=0$.\\
Again we have to show that
\[
\dim_K\bigl(S_1g\cap P_3\bigr)
+
\dim_K\bigl(S_1h\cap(P+(g))_3\bigr)
\leq 5.
\]
In this case $\dim_K(S_1g\cap P_3)=0$, so it remains to prove that
$\dim_K(S_1h\cap(P+(g))_3)\leq 5$. Assume, to the contrary, that this dimension is $6$. Hence
\[
x_i h=p_i+\ell_i g,
\]
where $p_i\in P$ and $\ell_i\in S_1$ for all $1\leq i\leq 6$. We show that $\{\ell_1,\dots,\ell_6\}$ is linearly independent. If $\sum_{i=1}^6\alpha_i\ell_i=0$, then
\[
\left(\sum_{i=1}^6\alpha_i x_i\right)h\in P,
\]
and the assumption implies that $\alpha_i=0$ for all $i$.

Define an isomorphism $T:S_1\longrightarrow S_1$ by $T(x_i)=\ell_i$. Since $K$ is algebraically closed, $T$ has an eigenvector $0\neq u=\sum_{i=1}^6\beta_i x_i$ with $T(u)=cu$ for some $c\neq 0$. Therefore,
\begin{align*}
uh
&=\sum_{i=1}^6\beta_i p_i
  +g\sum_{i=1}^6\beta_i\ell_i\\
&=\sum_{i=1}^6\beta_i p_i+cug,
\end{align*}
so $u(h-cg)\in P$. Since $g$ and $h$ are linearly independent modulo $P$, the element $h-cg$ is nonzero in $A_2$, contradicting the assumption. 
\end{proof}

\begin{thm}
Let $S = K[x_1, \dots, x_6]$, and let $I =(f_1, \dots, f_6, g ) \subset S$ be an almost complete intersection generated by quadrics, where $f_1, \dots, f_6$ form a regular sequence and $g \in S_2$. Then $I$ satisfies the Eisenbud--Green--Harris conjecture.
\end{thm}
\begin{proof}
    Set $P=(f_1,\dots,f_6)$ and $A=S/P$. We need to show that for every $0\leq t\leq 6$, there exists an ideal $L\supseteq (x_1^2,\dots,x_6^2)$ such that $|I_t|=|L_t|$ and $|L_{t+1}|\leq |I_{t+1}|$. The cases $t=0,1,6$ are trivial. For $t=0,1$ we take $L=(x_1^2,\dots,x_6^2)$. For $t=6$ we take $L$ to be the ideal generated by $x_1^2,\dots,x_6^2$ together with all monomials in $S_6$, since $I_6=S_6$. 
    
    As in the proof of Theorem~\ref{5}, we may assume that neither $I$ nor $(P:I)$ contains a reducible quadric. Then $(P:I)_1=0$; otherwise, a nonzero linear form in $(P:I)$ multiplied by any nonzero linear form would give a reducible quadric in $(P:I)_2$. Hence $\mathcal{H}_{S/(P:I)}(1)=6$. By Proposition~\ref{6},
    \[
    \mathcal{H}_{S/I}(5)=\mathcal{H}_{A}(5)-6=0.
    \]
    Thus $I_5=S_5$. Hence, for $t=4$ we may take any monomial ideal $L$ containing $x_1^2,\dots,x_6^2$ with $|L_4|=|I_4|$ and for $t=5$, we choose $L$ to be the ideal generated by $x_1^2,\dots,x_6^2$ together with all monomials in $S_5$.

    If $t=2$, then the assertion follows from Proposition~\ref{4}, applied with $d=2$.
    Assume that $t=3$. Since $(P:I)$ contains no reducible quadric, we have
    $S_1g\cap P_3=0$; otherwise, a nonzero $\ell\in S_1$ with $\ell g\in P$ would belong to $(P:I)_1$ and would produce a reducible quadric in $(P:I)_2$. Therefore,
    \[
    I_3=P_3\oplus S_1g
    \]
    and $|I_3|=36+6=42$. Let $L$ be the ideal generated by $x_1^2,\dots,x_6^2$ and the monomials \[x_1x_2x_3,x_1x_2x_4,x_1x_2x_5,x_1x_2x_6,x_1x_3x_4,x_1x_3x_5\,.\]
    Clearly, $|L_3|=42$. Note that
    \begin{align*}
        |L_4|&=|(x_1^2,\dots,x_6^2)_4|+9\\
        &=|S_4|-{6 \choose 4}+9\\
        &={9 \choose 5}-{6 \choose 4}+9\\
        &=120.
    \end{align*} 
  We show that $|I_4|\geq 120$. Assume, to the contrary, that $|I_4|\leq 119$. Thus 
  \[
  |(f_1,\dots,f_6)_4|+|S_2g|-|S_2g\cap (f_1,\dots,f_6)_4|\leq 119
  \]
  i.e. $|S_2g\cap P_4|\geq 13$. Since $P:I=P:g$, multiplication by $g$ identifies $(P:I)_2$ with $S_2g\cap P_4$. Hence $|(P:I)_2|\geq 13$. 

It is well known that the locus of reducible quadrics in 
$\mathbb{P}^{\binom{n+1}{2} - 1}$— that is, the projective space of all quadratic forms in $n$ variables— 
forms a closed algebraic subset of dimension $2n-2$.
This subset corresponds to all quadrics that factor as a product of two linear forms.

In particular, for $n = 6$, this locus has dimension $10$.
Therefore, any linear subspace of $\Sym^2(K^6)$ of dimension at least $11$ must intersect the reducible locus, meaning it necessarily contains at least one reducible quadric. 

This contradicts the assumption that $(P:I)_2$ contains only irreducible quadrics, and completes the proof.
\end{proof}

\end{document}